\documentclass[11pt]{amsart}
\usepackage{mathrsfs}
\usepackage{amsmath}
\usepackage{amscd}
\usepackage{amsmath}

\usepackage{xypic}
\usepackage{mathrsfs}
\usepackage{amsfonts}
\usepackage[colorlinks,linkcolor=blue,citecolor=blue, pdfstartview=FitH]{hyperref}

\numberwithin{equation}{section}

\theoremstyle{plain}
\newtheorem{thm}{thm}[section]
\newtheorem{theorem}[thm]{Theorem}
\newtheorem{lemma}[thm]{Lemma}

\newtheorem{proposition}[thm]{Proposition}

\theoremstyle{definition}

\newtheorem{remark}[thm]{Remark}

\newtheorem{definition}[thm]{Definition}

\newtheorem{defn-thm}[thm]{Definition-Theorem}

\newcommand{\shO}{\mathcal{O}}
\newcommand{\gr}{\text{gr}}
\newcommand{\dbar}{\overline{\partial}}
\newcommand{\pa}{\partial}

\newcommand{\Id}{\text{Id}}
\newcommand{\DR}{\text{DR}}

\newcommand{\MHM}{\text{MHM}}
\newcommand{\HM}{\text{HM}}
\newcommand{\h}{\text{H}}
\newcommand{\snc}{\text{SNC}}

\usepackage{fancyhdr}
\renewcommand{\leftmark}{{\scriptsize On the Kodaira--Saito vanishing theorem of weakly ample divisor}}

\begin{document}
\title{On the Kodaira--Saito vanishing theorem of weakly ample divisor}
\makeatletter
\let\uppercasenonmath\@gobble
\let\MakeUppercase\relax
\let\scshape\relax
\makeatother

\author{Yongpan Zou}

\address{Yongpan Zou, Graduate School of Mathematical Science, The University of Tokyo, 3-8-1 Komaba, Meguro-Ku, Tokyo 153-8914, Japan} \email{598000204@qq.com}


\maketitle

\begin{abstract}
On smooth projective variety, for a reduced effective divisor which is weakly ample in the sense of cohomology, we introduce a Kadaira--Saito vanishing theorem for it. 
\end{abstract}

\maketitle

\section{Introduction}

One important topic in algebraic geometry is the vanishing of the cohomology group of sheaves. The classical Cartan--Serre--Grothendieck theorem (see \cite[Theorem 1.2.6]{Laz}) says that a line bundle $L$ on compact complex manifold $X$ is ample if and only if for any coherent sheaf $\mathcal F$, there exist a positive integer $m_0=m_0(X,\mathcal F,L)$ such that:
$$
 \h^{q}(X,\mathcal F\otimes L^{\otimes m})=0, \quad \text{for any} \quad q > 0,\quad m\geq m_0.
$$
On the other hand, the ampleness is equivalent to the positivity of the metrical curvature on a compact complex manifold, due to the Kodaira embedding theorem. From the point of view of the cohomology group, we can generalize the concept of ampleness.

\begin{definition}
A line bundle $L$ over a smooth projective variety or compact K\"ahler manifold $X$ is called \emph{$k$-ample}, if for any coherent sheaf $\mathcal{F}$, there exists a positive integer $m_0=m_0(X,L,\mathcal{F})$  such that:
$$
\h^q(X,\mathcal{F}\otimes L^{\otimes m})=0, \quad \text{for any} \quad q > k,\quad
m\geq m_0.
$$
\end{definition} \label{defn1}
A divisor $D$ is called \emph{$k$-ample} if $\mathcal{O}_X(D)$ is a $k$-ample line bundle. The $0$-ample is equivalent to the usual definition of ampleness. For the usual ample line bundle $L$ and the canonical bundle $K_X$ on $X$, we have the Kodaira vanishing as follows:
\begin{align*} 
    \h^q(X,K_X \otimes  L)=0,\quad \text{for}\quad q > 0.
\end{align*}
It is natural to ask if the Kodaira vanishing theorem holds for $k$-ample line bundle $F$, i.e., do we have the next?
\begin{align}\label{kodai}
    \h^q(X,K_X \otimes F)=0,\quad \text{for}\quad q > k.
\end{align}
However, for the $k$-ample line bundle, the Kodaira vanishing is not true anymore, even when $X$ is projective. Ottem constructs a counterexample in \cite[Section $9$]{Ott}, there exists a projective threefold $X$ and a 1-ample line bundle L such that $H^2(X, K_X\otimes L) \neq 0$. If we assume moreover $F=\mathcal O_X(D)$ for some simple normal crossing divisor ($\snc$ for simplicity), then the above Kodaira vanishing theorem \eqref{kodai}, and the logarithmic Akizuki--Nakano vanishing theorem are both valid, see \cite{EsVi86, Liu, Liu1} for more details.
For a reduced effective divisor, the Hodge ideal introduced by Musta\c{t}\u{a}--Popa in \cite{MP19} is a powerful tool to investigate the singularity of divisors. We find the next vanishing theorem holds:
\begin{theorem} \label{main1}
    Let $X$ be a smooth projective variety, and $D$ is a reduced effective divisor on $X$. If $D$ is the support of some effective $k$-ample divisor, then one have
\begin{equation*}
\rm H^q(X,K_X \otimes \mathcal{O}_X(D) \otimes \mathcal{J}((1-\epsilon)D))=\rm H^q(X,K_X \otimes \mathcal{O}_X(D) \otimes \mathcal I_0(D))=0,\quad \text{for}\quad q > k.
\end{equation*}
\end{theorem}
\noindent Here $\mathcal J((1-\epsilon)D)$ is the multiplier ideal sheaf associated to the $\mathbb Q$-divisor $(1-\epsilon)D$ with $0 < \epsilon \ll 1$, and $\mathcal I_0(D)$ is the so called \emph{$0$-th Hodge ideal}. Moreover, one have $\mathcal J((1-\epsilon)D)=\mathcal I_0(D)$.

According to the definition of weakly ample divisor, it is natural to consider the germs of meromorphic functions $ \mathcal O_X(\ast D)= \bigcup_{k \ge 0} \mathcal{O}_X (kD)$ with poles along $D$. In Morihiko. Saito's mixed Hodge theory, this is a left $\mathcal D$-module underlies the mixed Hodge module $j_{\ast} \mathbb Q_U^{\rm H} [n]$, where $U = X \setminus D$ and $j: U \hookrightarrow X$ is the inclusion map. It, therefore, comes
with an attached Hodge filtration (increasing) $F_k \mathcal{O}_X(*D)$. The associated right $\mathcal D$-module is denoted by $K_X(\ast D):= K_X\otimes \mathcal O_X(\ast D)$. Saito shows that this filtration is contained in the pole order filtration, namely 
\begin{equation*}
F_k \mathcal{O}_X(*D) \subseteq \mathcal{O}_X \big( (k+1) D\big) \,\,\,\,\,\, {\rm for ~all} \,\,\,\, k \ge 0,
\end{equation*}
and
\begin{equation*}
F_{k-n} K_X(*D) \subseteq K_X \big( (k+1) D\big) \,\,\,\,\,\, {\rm for ~all} \,\,\,\, k \ge 0.
\end{equation*}
The inclusion above leads to defining for each $k \ge 0$ a coherent sheaf of ideals $I_k (D)$ by the formula
$$F_k \mathcal O_X(\ast D) = \mathcal{O}_X \big( (k+1) D\big)\otimes \mathcal I_k(D).$$
$$F_{k-n} K_X(\ast D) = K_X \big( (k+1) D\big)\otimes \mathcal I_k(D).$$
These are the definition of $k$-th Hodge ideals.
From the point of view of the mixed Hodge module, we shall study the Saito-type vanishing theorem for weakly ample divisors. Firstly we introduce the famous Saito vanishing for ample line bundle, it is a far generalization of Kodaira's vanishing theorem. Recently, there are $L^{2}$-approaches to the Saito Vanishing Theorem, see \cite{DH19} and \cite{Kim23} for details.
\begin{theorem}[Saito vanishing thorem]  \cite[Theorem $2.1$]{Schn16}\label{saito}
Let $X$ be a complex projective variety and let $\mathcal{M} \in \rm{MHM}(X)$ be a graded polarizable mixed Hodge module on $X$. For every ample line bundle $L$ on $X$, we have
		$$ \mathbb{H}^{k}\big( X, \rm gr^{F}_p \rm DR(\mathcal{M}) \otimes L \big) = 0 \qquad \text{for all} \quad k > 0, \quad p \in \mathbb{Z}.$$
\end{theorem}
\noindent Here $\DR(\mathcal M)$ is the de Rham complex of $\mathcal M$:
$$  \DR_{X}(\mathcal M) = \left [\mathcal M \to \Omega_{X}^{1} \otimes\mathcal M \to \cdots \to \Omega_{X}^{n} \otimes \mathcal M\right] [n].$$
On this note, we established the following vanishing theorem for weakly ample divisors.

\begin{theorem}[Main Theorem] \label{main2}
		Let $X$ be a smooth projective variety and $D$ be a $k$-ample effective Cartier divisor on $X$. Suppose that $U = X \setminus D$ and consider the open inclusion $j \colon U \hookrightarrow X$. Let $\mathcal M$ be a variation of Hodge structures on $U$ that we identify with the corresponding pure polarizable Hodge module. We have the following vanishing:
	\begin{align*}
		&\mathbb H^{l}(X, \rm{gr}^F_{p}\rm DR_{X}(j_{\ast} \mathcal M)) = 0 
		\end{align*}
		for all $l > k$, $p \in \mathbb Z$.
	\end{theorem}
The above Theorem \ref{main1} is the direct corollary of our main Theorem \ref{main2} when
 $\mathcal M = \mathbb Q_U^{\rm H} [n]:= (\mathcal O_U, F, \mathbb Q_U[n])$ associated to the trivial variation of Hodge structure $\mathbb Q$ on $X$. Indeed, in this case, we have $j_{\ast} \mathcal M=\mathcal O_X(\ast D)$ and $\gr^F_{-n} \DR (\mathcal O_X(\ast D)) \simeq K_X \otimes F_0\mathcal{O}_X(\ast D)= K_X\otimes \mathcal{O}_X(D) \otimes \mathcal{I}_0(D)$. Moreover, if the divisor is SNC, we can calculate each $\rm{gr}^F_{p}\rm DR_{X}(j_{\ast} \mathcal M)$ and deduce the logarithmic Akizuki--Nakano vanishing theorem for weakly ample divisor as in \cite{Liu}.

%
%

\begin{remark}
 There is an original differential geometric proof of the
vanishing theorem by using positive metrics on ample
line bundles. There is also a natural generalization for positive
line bundles. A line bundle $L$ over a compact complex manifold $X$
is called \emph{$k$-positive}, if there exists a smooth Hermitian
metric $h$ on $L$ such that the Chern curvature
$-\sqrt{-1}\pa\dbar\log h$ has at least $(n-k)$ positive
eigenvalues at any point on $X$. In $1962$, Andreotti and Grauert
proved in \cite[Theoreme 14]{AG62} that $k$-positive line bundles
are always $k$-ample (see also \cite[Proposition 2.1]{DPS}). In
\cite{DPS}, Demailly, Peternell and Schneider proposed a converse to
the Andreotti-Grauert theorem and asked whether $k$-ample line
bundles are $k$-positive. In some special cases this is true, see \cite{Dem11, Mat, Yang} for more details. However, there exist higher dimensional counterexamples in the range
$\frac{\dim X}{2}-1<q<\dim X-2$, constructed by Ottem in \cite{Ott}. For the vanishing theorem of $k$-positive line bundle, the reader can also consult \cite[Proposition $5.4$]{Mat14}.
\end{remark}

\noindent\textbf{Acknowledgements}:
The author would like to thank his advisor Professor Shigeharu. Takayama for guidance and encouragement. The author is supported by the Japan Society for the Promotion of Science.

\section{Preliminaries}
In this section, we introduce some standard definitions and facts. First we let $X$ be a compact complex manifold with dimension $\dim_{\mathbb{C}}X=n$, with simple normal crossing divisor $D=\sum_{i=1}^r D_i$. For every point $p\in X$, we can choose a coordinate neighborhood of $p$ such that locally $D =\{\prod_{i=1}^k z_i = 0\}$. We may focus on the sheaves of germs of $p$-forms which have logarithmic poles along the divisor $D$, we denote it by $\Omega^p_X(\log D)$. On the open set $V$ on $X$, sections of $\Omega_X^1(\log D)$ are given by:
\begin{align*}
\Omega_X^1(\log D):= \mathcal{O}_X \frac{dz_1}{z_1} \oplus \cdots \oplus \frac{dz_k}{z_k} \oplus \mathcal{O}_X dz_{k+1} \oplus\cdots\oplus\mathcal{O}_X dz_{n}.
\end{align*}

\noindent This describes the holomorphic vector bundle $\Omega_X^1(\log D)$ on $X$. Then, for any $k\geq 0$, the vector bundle $\Omega_X^k(\log D)$ is the $k$-th exterior power,
\begin{align*}
\Omega_X^k(\log D):= \bigwedge^k \Omega_X^1(\log D).
\end{align*}
The next notation
\begin{align*}
    \Omega^p_X(*D) := \bigcup_{k\geq 0}\Omega^p_X(k D) = \lim_{\rightarrow} \Omega^p_X(kD).
\end{align*}
represent the germs of meromorphic $p$-forms, which is holomorphic on $X\setminus D$, and have arbitrary (finite) order poles along $D$.
If $D'$ be another divisor with the same supports as $D$, i.e., $\text{Supp}(D')=D$, we have
\begin{align}\label{1.0}
\Omega^p_X(*D)=\Omega^p_X(*D').
\end{align}

\noindent We say a SNC divisor $D$ is the support of some effective $k$-ample divisor, if there exists a $k$-ample divisor $D'$ such that $\text{supp}(D')=D$.


We can only define the logarithmic forms for SNC divisors. To a reduced effective divisor $D$ on $X$, one associates the left $\mathcal{D}$-module of functions with poles along $D$,
$$\mathcal{O}_X (*D) = \bigcup_{k \ge 0} \mathcal{O}_X (kD),$$
i.e. the localization of $\mathcal{O}_X$ along $D$. The association right $\mathcal D$-module is $K_X(\ast D)$.

\begin{lemma} \cite[Lemma 2.1]{Liu}\label{mere}
Let $X$ be a compact complex manifold with dimension $\dim_{\mathbb{C}}X=n$, and $D$ be an effective $k$-ample divisor. Let $\mathcal{F}$ be arbitrary vector bundle on $X$, we have:
\begin{align*}
H^q(X,\mathcal{O}_X(\ast D)\otimes \mathcal{F})=0, \quad \text{when}\quad q>k.
\end{align*}
\end{lemma}

\begin{proof}
We use $[\bullet]$ and $[\bullet]_{\ell}$ to represent the class in $H^q(X,\mathcal{O}_X(\ast D)\otimes \mathcal{F})$ and $H^q(X,\mathcal{O}_X(\ell D)\otimes \mathcal{F})$. If $[\alpha]\in H^q(X,\mathcal{O}_X(\ast D)\otimes \mathcal{F})$, by the definition, there exist a integer $\ell_0> 0$ such that
\begin{align}\label{1.1}
    [\alpha]_{\ell}\in H^q(X,\mathcal{O}_X(\ell D)\otimes \mathcal{F}),\quad  \text{for} \quad \ell\geq \ell_0.
\end{align}
Since $\mathcal{O}_X(D)$ is $k$-ample, according to definition, there exist a integer $\ell_1$ such that
\begin{align}\label{1.2}
H^q(X,\mathcal{O}_X(\ell D)\otimes \mathcal{F})=0,\quad \text{for} \quad q>k,\quad
\ell\geq \ell_1.
\end{align}
 We choose $\ell \geq \max\{\ell_0, \ell_1\}$ and $q>k$, we know
\begin{align*}
\alpha=\dbar\beta
\end{align*}
for some $\beta\in A^{0,q-1}(X,\mathcal{O}_X(\ell D)\otimes \mathcal{F})\subset
A^{0,q-1}(X,\mathcal{O}_X(*D)\otimes \mathcal{F})$. Therefore when $q>k$, we know
\begin{align*}
H^q(X,\mathcal{O}_X(*D)\otimes \mathcal{F})=0.
\end{align*}
\end{proof}

\begin{remark} \label{keyrem}
    If $X$ is a smooth projective variety, let $D$ be an effective divisor on $X$ and set $U:=X\setminus D$. For any vector bundle $\mathcal{F}$ on $X$, we have (cf, \cite[Proposition $5.1$]{Ott})
    $$ H^i(U, \mathcal F|_U) = \lim_{\rightarrow} H^i(X, \mathcal{F}(m D)).
    $$
    Given a log resolution $f: Y \to X$ of the pair $(X, D)$ which $E:= (f^{\ast}D)_{\rm red}$ and set $V:=f^{-1}(U)$ and the inclusion $i: U\to Y$. Since $U$ and $V$ is isomorphism, we have
    \begin{align*}
        H^i(U, \mathcal F|_U) &\simeq H^i(V, \mathcal F|_V) = \lim_{\rightarrow} H^i(Y, \mathcal{F}(m E))\\
        &= H^i(Y, \lim_{\rightarrow}\mathcal{F}(m E))= H^i(Y, \mathcal{F}(\ast E))
    \end{align*}
Then if we assume that $D$ be a $k$-ample divisor, we have the vanishing 
 $$H^i(Y, \mathcal{F}(\ast E))=0$$
 for $i > k$.
\end{remark}
Next, we introduce Deligne's canonical extension of flat vector bundles.
This is a local problem, we can assume $X$ be a poly-disk $\Delta^{n}$ and let $U$ be $(\Delta^{\ast})^r \times \Delta^{n-r}$. Let $t_1,\cdots,t_n$ be the coordinate system on $X$, and let $D := X\setminus U$ be the simple normal crossing divisor, here $D_j=\{ t_j = 0\}, 1\leq j \leq r$ be the hypersurface. For convenience, we consider the case when $r=n$. Let $\mathcal{V}$ be the holomorphic vector bundle on $U$ with flat connection $\nabla: \mathcal{V} \to \Omega^1_U \otimes \mathcal{V}$, i.e., $\nabla\circ\nabla =0$. It is easy to see that the fundamental group of $U$ is isomorphic to $\mathbb{Z}^n$, and the fundamental group act on $\mathcal{V}$ by the parallel transport. The kernel of this flat connection $\mathcal{H}=\ker \nabla$ is a local system. The local system $\mathcal{H}$ has monodromy operators $T_1, \cdots, T_n$ with $T_j$ given by moving in a counterclockwise direction around the hyperplane $t_j=0$. We say $(\mathcal{V}, \nabla)$ has quasi-unipotent monodromy if each $T_j$ is quasi-unipotent
(i.e., $(T_j^p-\Id)^q=0$ for some $p,q\in \mathbb N$). Let $\alpha \in \mathbb R^n$ with $\alpha= (\alpha_1, \cdots, \alpha_n)$, we set the subset $I_{\alpha}$ of $\mathbb R^n$ to be $I_{\alpha}:= [-\alpha_1,-\alpha_1+1)\times \cdots \times [-\alpha_n,-\alpha_n+1)$.
Deligne proved the following classical result:

\begin{theorem} \label{Deligne}
Let $\mathcal{V}$ be the holomorphic vector bundle on $U$ with flat connection $\nabla$, and $(\mathcal{V},\nabla)$ have quasi-unipotent monodromy. Then for any $I_{\alpha}$, there exist the extension vector bundle $\mathcal{V}_{\alpha}$ of $\mathcal{V}$ on $X$ with the connection $\nabla$ extended on $X$ as
 $$
 \nabla: \mathcal{V}_{\alpha} \to \Omega^1_{X}(\log D) \otimes \mathcal{V}_{\alpha}
 $$
only have a logarithmic pole, and for each $j$, the eigenvalue of the residue operator
\[ \rm Res_{D_j}\nabla := t_j\nabla_{\frac{\partial}{\partial t_j}}: \mathcal{V}_{\alpha}|_{D_j} \to \mathcal{V}_{\alpha}|_{D_j} \]
include in $[-\alpha_j,-\alpha_j+1)$.
\end{theorem}


If moreover, $\mathcal{V}$ is a complex polarized variation of Hodge structure on $U$, it carries a Hodge filtration $F^{\bullet}$. Now we do not need the quasi-unipotent assumption because each eigenvalue has norm $1$, see \cite[Lemma $4.5$]{Sch73}. Noting Schmid's original paper assumes that we work with an integral variation of Hodge structure, for complex VHS, there are more recent papers like \cite{Den22} and \cite{SS22}. We consider the Deligne extension $\mathcal{V}_{\alpha}$, for $\alpha \in \mathbb{R}^r$, whose eigenvalues of the eigenvalues of the residue along $D_i$ lie inside the  $[-\alpha_i,-\alpha_i+1)$. By the following Schmid's nilpotent orbit theorem for arbitrary complex polarized Hodge structures, the graded pieces $\mathcal{V}^{p,q}$ also extend to vector bundle $\mathcal{V}^{p,q}_{\alpha}$ on $X$.

\begin{theorem}\cite{Den22}
		Let $X$ be a complex manifold and let $D = \sum_{i=1}^{r} D_{i}$ be a simple normal crossing divisor on $X$. Let $\mathcal{V}$ be a complex polarized variation of Hodge structures on $U = X \setminus D$. Then for every $\alpha \in \mathbb{R}^{r}$, $F^{p} \mathcal{V}_{\alpha}$ and $\mathcal{V}_{\alpha}^{p,q} = F^{p}\mathcal{V}_{\alpha}/ F^{p+1} \mathcal{V}_{\alpha}$ are locally free sheaves.
	\end{theorem}

\noindent Note that the connection $\nabla \colon \mathcal{V} \to \Omega_{U}^{1} \otimes \mathcal{V}$ extends to a logarithmic connection
	$$ \nabla \colon \mathcal{V}_{\alpha} \to \Omega_{X}^{1} (\log D) \otimes\mathcal{V}_{\alpha}$$
	and satisfies the Griffiths transversality condition, i.e., on each graded piece $\mathcal{V}_{\alpha}^{p,q}$, the connection $\nabla$ induces an $\mathcal{O}_{X}$-linear map
	$$ \theta \colon \mathcal{V}_{\alpha}^{p,q} \to \Omega_{X}^{1} (\log D) \otimes \mathcal{V}_{\alpha}^{p-1,q+1}.$$

We denote the associated de Rham complex of $\mathcal{V}_{\alpha}$ by $\DR(\mathcal{V}_{\alpha}, \log D)$ as follows:
\begin{align}
\quad \big[\mathcal{V}_{\alpha} \rightarrow \mathcal{V}_{\alpha} \otimes\Omega^1_{X}(\log D) \rightarrow \cdots \rightarrow \mathcal{V}_{\alpha}\otimes\Omega^n_{X}(\log D) \big].
\end{align}
And denote its $p$-graded piece complex by $\gr_p\DR(\mathcal{V}_{\alpha}, \log D)$ as follows:
\begin{align}
\quad \big[\mathcal{V}^{p,q}_{\alpha} \rightarrow \mathcal{V}^{p-1,q+1}_{\alpha} \otimes\Omega^1_{X}(\log D) \rightarrow \cdots \rightarrow \mathcal{V}^{p-n,q+n}_{\alpha}\otimes\Omega^n_{X}(\log D) \big].
\end{align}

	
We now review some necessary facts on Hodge modules. All of the results can be found in two original papers of M. Saito \cite{Sai88, Sai90}. The reader also can consult Christian. Schnell's many excellent papers, surveys, and notes on his website, for example, \cite{Sch19}.
On smooth complex variety $X$, Saito defines two categories $\HM(X, w)$ and $\MHM(X)$ which are the category of pure Hodge modules with weight $w$ and the category of graded polarizable mixed Hodge modules. The objects in these categories are certain holonomic filtered $\mathcal D$-modules with an extra structure that is required to satisfy suitable conditions. For a singular variety $X$, we embed $X$ in a smooth variety $Z$ and define $\MHM(X)$ as the category of mixed Hodge modules on $Z$ whose support is in $X$. The category $\MHM(X)$ does not depend on the choice of closed embedding $X \hookrightarrow Z$. We can similarly define $\HM(X, w)$ as well. 
The two important filtrations on a mixed Hodge module are the weight filtration $W_{\bullet}$ and the Hodge filtration $F_{\bullet}$, which generalize the weight filtration and the Hodge filtration in the theory of mixed Hodge structures.
Firstly, we have the associated de Rham complex of $\mathcal M$:
$$  \DR_{X}(\mathcal M) = \left [\mathcal M \to \Omega_{X}^{1} \otimes\mathcal M \to \cdots \to \Omega_{X}^{n} \otimes \mathcal M\right] [n].$$
The filtration on the de Rham complex of a mixed Hodge module $\mathcal M$ on a smooth variety $X$ is given by
		$$ F_{k} \DR_{X}(\mathcal M) = \left[F_{k}\mathcal M \to \Omega_{X}^{1} \otimes F_{k+1}\mathcal M \to \cdots \to \Omega_{X}^{n} \otimes F_{k+n}\mathcal M\right] [n].$$
		After taking the graded pieces, we obtain a complex of $\shO_{X}$-modules
		$$ \gr_{k}\DR_{X}(\mathcal M) = \left[ \gr_{k} \mathcal M \to \Omega_{X}^{1} \otimes \gr_{k+1} \mathcal M \to \cdots \Omega_{X}^{n} \otimes \gr_{k+n} \mathcal M\right][n]. $$
		We can extend the notion of graded de Rham complexes to objects in the derived category of mixed Hodge modules. 
%
%
%
%
%
%
%
%
%
 If $j: U \to X$ is a locally closed immersion between smooth varieties and $\mathcal M \in \MHM(U)$, then the transformation $j_{!}\mathcal M \to j_{\ast} \mathcal M$ induces a map $\mathcal H^{0}(j_{!} \mathcal M) \to \mathcal H^{0}(j_{\ast} \mathcal M)$. The image of this map is the intermediate extension of $\mathcal M$ which is denoted by $j_{!\ast} \mathcal M$. Furthermore, let's suppose that $U \subset X$ is open and $D = X \setminus U$ is a divisor. Then $j_{\ast}$ and $j_{!}$ are induced by an exact functor at the level of abelian categories. In other words, if $\mathcal M \in \MHM(U)$ is a mixed Hodge module on $U$, then we can view $j_{\ast} \mathcal M$ and $j_{!} \mathcal M$ as objects inside $\MHM(X)$. Moreover, the underlying $\mathcal O_{X}$-module structure of $j_{\ast} \mathcal M$ agrees with the usual sheaf theoretic push-forward.
Lastly, we consider proper morphisms $\pi: Y \to X$, taking the graded de Rham complex commutes with taking the push forward. In other words, if $\mathcal M$ is a mixed Hodge module on $Y$, we have
	 $\gr_{p}\DR_{X} (\pi_{\ast} \mathcal M) \simeq \mathbf{R} \pi_{\ast}( \gr_{p}\DR_{Y} \mathcal M).$

\section{The proof of the main theorem}
Now we begin to explain Theorem \ref{main2}. In the first subsection, we assume $D$ is a simple normal crossing divisor.
\subsection{the quasi-isomorphism of complexes}
In this subsection, we will explain the following theorem.
\begin{theorem} \label{snc-main}
Let $D$ be a $k$-ample simple normal crossing divisor on a smooth projective variety $X$. Let $\mathcal V$ be the complex polarized variation of Hodge structure on $U=X\setminus D$ with filtration $F^{\bullet}$. And we assume all $\alpha_i\geq 0$, we have
\begin{equation}
    \mathbb{H}^{p+q} (X, \rm gr^F_{-q} \rm DR(\mathcal{V}_{\alpha}, \log D)) =0 
\end{equation}
 for all $p+q > n+k$.   
\end{theorem}

Now we will recall the classical Deligne's extension, the material is standard.
Let $E$ be the holomorphic vector bundle on the open sub-manifold $U$ of $X$, and it has a flat connection $\nabla$. Moreover, the boundary $D=X-U$ is a simple normal crossing divisor and the connection $\nabla$ is quasi-unipotent along the divisor $D$. 

Under this assumption, by Theorem \ref{Deligne}, there is a canonical extension $E_{\alpha}$ of $E$ on $X$ which the eigenvalue of residue operator included in $[-\alpha, -\alpha+1)$, and the extension connection $\nabla$ has at most logarithmic poles,
$$
\nabla : E_{\alpha} \rightarrow E_{\alpha} \otimes \Omega^1_{X}(\log D).
$$
Moreover, the residue of the connection $\nabla$ along the divisor $D$ is quasi-unipotent. Note that the unipotent means all eigenvalues are $0$.
We now consider some complexes of quasi-coherent sheaves, the first is
\begin{align}
\DR(E): \quad E \rightarrow E\otimes\Omega_U^1 \rightarrow E\otimes\Omega_U^2 \rightarrow \cdots \rightarrow E\otimes\Omega_U^n,
\end{align}
this is the complex vector bundle on $U$. By the inclusion $i: U \hookrightarrow X$, we can push forward complex $\DR(E)$ to $X$, and obtain a complex of quasi-coherent sheaves $\Xi^{\bullet}_{\alpha}$ on $U$. We know its $p$-th term is
$$\Xi^p = i_{\ast}(E\otimes\Omega_U^p)$$
for any $p \geq 0$, but the section of $\Xi^p$ may have the essential singularities along $D$, so we consider the sub-complexes of $\Xi^{\bullet}$ with mild singularities. Then we set
\begin{align}
\Xi^p_{\alpha}(*D) = E_{\alpha}\otimes\Omega_U^p(*D)
\end{align}
with terms only having poles along the divisor $D$.
Moreover, we set
\begin{align}
\Xi^p_{\alpha}(\log D)=  E_{\alpha}\otimes\Omega_U^p(\log D)
\end{align}
with terms only having logarithmic poles along the divisor $D$.

So now we have three complexes on $X$, i.e., $\Xi^{\bullet}, \Xi^{\bullet}_{\alpha}(*D), \Xi^{\bullet}_{\alpha}(\log D)$, the third is complex due to Deligne's extension theorem. All differentials in these complexes are induced from the connection $\nabla$. We now study the inclusion morphism between $\Xi^{\bullet}_{\alpha}(*D), \Xi^{\bullet}_{\alpha}(\log D)$.

For every point on $X$, we can find a small neighborhood which is isomorphism to $\Delta^n$, denote the local coordinate by $t_1,\cdots,t_n$. since $D$ is simple normal crossing, locally $D\cap \Delta$ be the zero locus of $t_1\cdots t_r=0$, we have
$$
U\cap \Delta^n = (\Delta^{\ast})^r\times \Delta^{n-r}.
$$
To simplify the calculation, we again assume $r=n$, the general case is similar.
On $\Delta^n$, the canonical extension $ E_{\alpha}$ can be generated by some sections.
Let $V$ be the generic fiber of vector bundle $E$, it is a finite-dimensional vector space. The local system has the monodromy $T_1,\cdots, T_n$, each $T_j$ represented by an anti-clockwise loop around the hyperplane $t_j=0$. According to the assumption before, each $T_j$ is the quasi-unipotent operator of $V$. For $\mathbf{\lambda} = (\lambda_1, \cdots, \lambda_n)$, each $\lambda_i$ is the eigenvalue of $T_i$.


Let $V$ be a vector space and let $T$ be a quasi-unipotent operator on $V$. Consider the decomposition
$$
T = T_s \cdot T_u
$$
where $T_s$ is semisimple and $T_u$ is unipotent. We know the smallest integer $k$ exists such that $T_s^k = \Id$.
Since $T_s$ is semisimple over $\mathbb{C}$, there is a decomposition
$$
V = \oplus V_{\alpha} ~~\text{with}~~  T_s|_{V_{\alpha}} = \lambda_{\alpha}\cdot \Id|_{V_{\alpha}},
$$
where $\lambda_{\alpha}$ are the eigenvalues of $T_s$. Since $\lambda_{\alpha}^k =1$, we can choose $\beta_{\alpha} \in (\alpha-1, \alpha]$ satisfying 
$$
\lambda_{\alpha} = e^{2\pi i \beta_{\alpha}}.
$$
We define $\log T_s$ by its action on eigenspaces by:
$$
\log T_s|_{V_{\alpha}} := 2\pi i \beta \Id|_{V_{\alpha}}.
$$
Since $T_u$ is unipotent, we can define its logarithmic as a convergent series
$$
\log T_u = - \sum_{m\geq 1} \frac{1}{m}(\Id - T_u)^m.
$$
Then we define the logarithmic of $T$ to be
$$
\log T := \log T_s + \log T_u, 
$$
where $\log T_u$ is nilpotent and $\log T_s$ is semisimple.
On the universal covering $\mathbb{H}^n$ of $\Delta^n$,
$$
\mathbb{H}^n \rightarrow (\Delta^{\ast})^n, \quad (z_1,\cdots,z_n) \rightarrow (e^{z_1},\cdots,e^{z_n}),
$$
We denote by $V^{\nabla}$ the space of multivalued flat sections of $(V, \nabla)$, which is a finite-dimensional $\mathbb{C}$ vector space.
Set $T_j$ to be the monodromy transformation with respect to $\gamma_j$, which pairwise commute and is an endomorphism of $V^{\nabla}$, that is, for any multivalued section $v(t_1, \cdots, t_n)\in V^{\nabla}$, one has 
$$
v(t_1, \cdots, t_j+1, \cdots, t_n) = (T_j v)(t_1, \cdots, t_n)
$$
and $[T_i, T_j]=0$ for any $i,j=1, \cdots, n$, set $R_j:= \frac{1}{2\pi i}\log T_j$.
Every element $v\in V^{\nabla}$ defines a holomorphic map: $\tilde{s}: \mathbb{H}^n \rightarrow V$ as follows,
$$
\tilde{s}(z) = e^{\sum (-z_j) R_j}v.
$$
Each $\tilde{s}$ descends to a holomorphic section $s$ of $E$ on $(\Delta^{\ast})^n$, and these sections generate the vector bundle $E_{\alpha}$ on $\Delta^n$. We can calculate the connection of these sections. Indeed, we have
$$
\nabla \tilde{s} = \sum_{k=1}^n (-R_k) e^{\sum (-z_j) R_j}v \otimes dz_k,
$$
and since $ dz_k= d \log t_k$, we obtain
\begin{equation} \label{deri}
\nabla s = \sum_{k=1}^n (-R_k) s \otimes d \log t_k =\sum_{k=1}^n (-\beta_k\cdot \Id - N_k) s \otimes d \log t_k.
\end{equation}
It is worth noting that $N_k s$ is the section corresponding to the vector $N_k v$.

For any section $\sigma$ of quasi-coherent sheaf $\Xi^p=i_{\ast}(\mathcal{V}\otimes\Omega_U^p)$, we write it as:
$$
\sigma = \sum_{I,\gamma} \sigma_I(\gamma) t^{\gamma}\otimes (d \log t)_I
$$
here $\sigma_I(\gamma)$ is the section of Deligne's canonical extension. Also $\gamma = (\gamma_1,\cdots,\gamma_n)\in \mathbb Z^n$ and $I$ over all subset of $\{1,2,\cdots,n\}$ of size $|I|=p$. We also use the convenient notation
$$
t^{\gamma} = \prod_{i=1}^n t_i^{\gamma_i}
$$
and
$$
(d \log t)_I = \prod_{i\in I}d \log t_i = \prod_{i\in I}\frac{dt_i}{t_i}.
$$
It is obvious that if $\sigma_I(\gamma)=0$ when $|\gamma|\ll 0$, then $\sigma$ is the section of $\Xi^{p}(*D)$.
If one have $\sigma_I(\gamma)=0$ when one of component $\gamma_i < 0$, then $\sigma$ is the section of $\Xi^{p}(\log D)$. From the formula \eqref{deri}, we can find that
$$
\nabla \sigma = \sum_{I,\gamma,k}(\gamma_k - \beta_k  - N_k) \sigma_I(\gamma)\cdot t^{\gamma} \otimes (d \log t_k \wedge (d \log t)_I).
$$
thus we write it as
$$
\nabla \sigma = \sum_{J,\alpha}\tau_J(\gamma) t^{\gamma}\otimes (d \log t)_J,
$$
Here $J$ is the set of all $(p+1)$-subset of $\{1,2,\cdots,n\}$, the coefficient is given by
\begin{equation}\label{diff}
\tau_J(\gamma) = \sum_{k\in J}(\gamma_k - \beta_k  - N_k) \sigma_{J\setminus \{k\}}(\gamma) (-1)^{\text{pos}(k,J)}.
\end{equation}
the notation $\text{pos}(k, J)$ means the position of $k$ in $J$, one advantage of the above formula is that the index $\gamma$ does not change under the derivative, therefore we can deal with each component one time.
It is also worth noting that $N_k$ is nilpotent, so the operator
$$
B_k = \gamma_k - \beta_k  - N_k
$$
is invertible if and only if $\gamma_k -\beta_k  \neq 0$.

To calculate the cohomology of these complexes, we consider the general abstract complex:
$$
M^{\bullet}: \quad M^0 \stackrel{\nabla}{\longrightarrow} M^1 \stackrel{\nabla}{\longrightarrow} M^2 \stackrel{\nabla}{\longrightarrow} \cdots \stackrel{\nabla}{\longrightarrow} M^n.
$$
We assume the element in $M^p$ is given by:
$$
\sigma = (\sigma_I)_{|I|=p},
$$
the index $I$ is the subset of $\{1,2,\cdots,n\}$ with size $p$. The differential operator $\nabla$ defined by:
$$
\nabla \sigma = (\tau_J)_{|J|=p+1},
$$
here
$$
\tau_J = \sum_{k\in J}(-1)^{\text{pos}(k,J)}B_k \sigma_{J\setminus \{k\}}.
$$
In this setting, $B_k$ can be any operator. the following lemma gives the condition that the complex $M^{\bullet}$ to be exact.

\begin{lemma}\label{exact}
If at least one $B_k$ is invertible, then the complex $(M^{\bullet}, \nabla)$ is exact.
\end{lemma}

\begin{proof}
After rearranging the index, we may assume that $B_1$ is invertible, we shall prove that this complex is contractible. A contract morphism $\epsilon: M^p \rightarrow M^{p-1}$ can be defined as follow:
$$
\epsilon(\sigma) = \big([1\not\in J] \cdot B_1^{-1}\sigma_{J\cup\{1\}})_{|J|=p-1}.
$$
the notation above $[P]=1$ when the statement $P$ is true, otherwise $[P]=0$. It is easy to calculate that $\nabla \epsilon + \epsilon \nabla = \Id$, this infers the desired result.
\end{proof}
In the complex $(\Xi^{\bullet}, \nabla)$ we discussed, the operator $B_k$ in \eqref{diff} is invertible when $\gamma_k- \beta_k \neq 0$. As the application of the above Lemma \ref{exact}, we know it is exact when $\gamma_k - \beta_k \neq 0$. 
In the three complexes on $X$, i.e., $\Xi^{\bullet}, \Xi^{\bullet}_{\alpha}(*D), \Xi^{\bullet}_{\alpha}(\log D)$ we discussed. As the application of the above Lemma \ref{exact}, the operator $B_k$ in \ref{diff} is invertible when $\gamma_k- \beta_k \neq 0$. If $\beta_k > -1$, then for all $\gamma_k \leq -1$, we have $\gamma_k- \beta_k \neq 0$. By our setup, $\beta_k \in (\alpha_k -1, \alpha_k]$, this is to say, we can let all $\alpha_k \geq 0$. In this assumption, we only need to consider the cases when $\gamma_k \geq 0$, these terms belong in all three complexes. So we conclude that the inclusion of complex between $\Xi^{\bullet}_{\alpha}(\log D) \hookrightarrow  \Xi^{\bullet}_{\alpha}(*D) \hookrightarrow \Xi^{\bullet}$ is quasi-isomorphism when all $\alpha_k\geq 0$. Now we come to the next lemma.

\begin{lemma} \label{quai}
    Two complexes $\Xi^{\bullet}_{\alpha}(\log D), \Xi^{\bullet}_{\alpha}(*D)$ are quai-isomorphism under the inclusion morphism, provided the $\alpha = (\alpha_1, \cdots, \alpha_r)$ satisfies all $\alpha_i \geq 0$.
\end{lemma}
	
If $\mathcal{V}$ is a variation of Hodge structure on $U$, it is a flat vector bundle that carries a Hodge filtration $F^{\bullet}$. We also consider the so-called Deligne extension $\mathcal{V}_{\alpha}$, for $\alpha \in \mathbb{R}^r$, whose eigenvalues of the residue along $D$ lie inside the $I_{\alpha}$. It de Rham complex $\DR(\mathcal{V}_{\alpha}, \log D)$ is:
\begin{align}
 \quad \big[\mathcal{V}_{\alpha} \rightarrow \mathcal{V}_{\alpha} \otimes \Omega^1_X(\log D) \rightarrow \cdots \rightarrow \mathcal{V}_{\alpha}\otimes\Omega^n_X(\log D) \big].
\end{align}	
By Schmid's nilpotent orbit theorem for arbitrary complex polarized Hodge structures, the graded pieces $\mathcal{V}^{p,q}$ also extend to vector bundle $\mathcal{V}^{p,q}_{\alpha}$ on $X$. We have the graded piece of the logarithmic de Rham complex $\gr^p \DR(\mathcal{V}_{\alpha}, \log D)$ as follows:
\begin{align}
 \quad \big[\mathcal{V}^{p,q}_{\alpha} \rightarrow \mathcal{V}^{p-1,q+1}_{\alpha} \otimes\Omega^1_{X}(\log D) \rightarrow \cdots \rightarrow \mathcal{V}^{p-n,q+n}_{\alpha}\otimes\Omega^n_{X}(\log D) \big].
\end{align}	
For the meromorphic forms, we also have the associated de Rham complex:
\begin{align}
\DR(\mathcal{V}_{\alpha}, \ast D): \quad \big[\mathcal{V}_{\alpha}\otimes \Omega^0_X(\ast D) \rightarrow \mathcal{V}_{\alpha} \otimes \Omega^1_X(\ast D) \rightarrow \cdots \rightarrow \mathcal{V}_{\alpha}\otimes\Omega^n_X(\ast D) \big].
\end{align}	

We now consider the complex induced by the above complex:
\begin{align} \label{complex}
H^q(X,\Omega^0_X(*D)\otimes \mathcal{V}_{\alpha}) \xrightarrow{\nabla} H^q(X,\Omega^1_X(*D)\otimes \mathcal{V}_{\alpha})\xrightarrow{\nabla}\cdots\xrightarrow{\nabla}
H^q(X,\Omega^n_X(*D)\otimes \mathcal{V}_{\alpha}).
\end{align}
The associated cohomology is denoted by:
 $$E_2^{p,q}=H^p_{\nabla}(H^q(X,\Omega^{\bullet}_X(*D)\otimes \mathcal{V}_{\alpha})).$$
 Due to the assumption of $k$ ampleness and Lemma \ref{mere}, we know
\begin{align}\label{$1.3$}
    E_2^{p,q}=0,\quad \text{for}\quad q>k.
\end{align}

\noindent Let $\mathbb{H}^{\bullet}(X, \Omega^{\bullet}(*D)\otimes \mathcal{V}_{\alpha})$ be the hypercohomology of complex $\DR(\mathcal{V}_{\alpha}, \ast D)$. According to \eqref{$1.3$}, we obtain
\begin{equation}
    \mathbb{H}^s(X, \Omega^{\bullet}(*D)\otimes \mathcal{V}_{\alpha})=0,\quad\text{when}\quad s> n+k.\label{1.4}
\end{equation}

\noindent On the other hand, by the quasi-isomorphism in Lemma \ref{quai}, when all $\alpha_i \geq 0$, we have,
\begin{equation*}
\mathbb{H}^s(X, \Omega^{\bullet}(*D)\otimes \mathcal{V}_{\alpha})\cong \mathbb{H}^s(X, \Omega^{\bullet}(\log D)\otimes \mathcal{V}_{\alpha}).
\end{equation*}
Hence combining (\ref{1.4}), we obtain
\begin{align}\label{1.41}
    \mathbb{H}^s(X, \Omega^{\bullet}(\log D)\otimes \mathcal{V}_{\alpha})=0,\quad\text{when}\quad s> n+k.
\end{align}
Finally, we know that Hodge filtration induces the spectral sequence:
\begin{align}
     E^{p,q}_1 = \mathbb H^{p+q}(X, \gr_{-q}\DR(\mathcal{V}_{\alpha}, \log D)) \Rightarrow \mathbb H^{p+q}(Y, \DR(\mathcal{V}_{\alpha}, \log D)).
\end{align}
Since $\mathcal{V}$ is a variation of Hodge structure, we know this spectral sequence degenerate at $ E_1$ page, and we conclude Theorem \ref{snc-main}.

\subsection{Saito type vanishing theorem}
In this section, we first recall Saito's ideal to build the relationship between the logarithmic de Rham complex and the usual de Rham complex. The author learns it in \cite{Kim23}.

Let $X$ be a smooth projective variety and let $D$ be an SNC divisor on $X$. Let $U$ be the complement $X \setminus D$ and denote the open inclusion as $j \colon U \to X$. We consider a variation of Hodge structures $\mathcal M$ on $U$. The main theorem involves the logarithmic de Rham complex of $\mathcal M$. We now describe the relationship between the logarithmic de Rham complex $\DR(\mathcal M_{\alpha}, \log D)$ and the de Rham complex of the $\mathcal D_{X}$-modules $j_\ast \mathcal M$. The results can be found in \cite{Sai90}. We define a subring of $\mathcal D_{X}$ by
	$$ V_{0}^{D} \mathcal D_{X} = \{ P \in \mathcal D_{X} : P \cdot (f)^{j} \subset (f)^{j} \text{ for all } j \}$$
	where $D$ is given by $f = 0$. If $(z_{1},\cdots, z_{n})$ is a system of coordinates such that $D$ is given by the equation $z_{1}\cdots z_{l} = 0$, then we can describe $V_{0}^{D} \mathcal D_{X}$ locally as
	$$V_{0}^{D}\mathcal D_{X} = \mathcal O_{X} \langle z_{1} \pa_{1},\cdots, z_{l} \pa_{l}, \pa_{l+1} , \cdots, \pa_{n} \rangle. $$
Giving a $V_{0}^{D} \mathcal D_{X}$-module structure on $\mathcal N$ is equivalent to giving a logarithmic connection
	$$ \nabla \colon \mathcal N \to \Omega_{X}^{1} (\log D) \otimes \mathcal N $$
	such that $\nabla^{2} = 0$. We can define the logarithmic de Rham complex of a $V_{0}^{D} \mathcal D_{X}$-module as follows:
	$$ \DR(\mathcal N, \log D) = \left[\mathcal N \to \Omega_{X}^{1} (\log D) \otimes \mathcal N \to \cdots \to \Omega_{X}^{n} (\log D) \otimes \mathcal N\right][n].$$
Here we shift $n$ to keep it the same as the usual notations. We have a filtration $F_{\bullet}$ on $V_{0}^{D} \mathcal D$ which is induced by the order filtration on $\mathcal D_{X}$. For a variation of Hodge structures $\mathcal M$ on $U$, the Deligne extension $\mathcal M_{\alpha}$ has a structure of filtered $V_{0}^{D}\mathcal D_{X}$-module by the logarithmic connection and the Hodge filtration which is constructed via the nilpotent orbit theorem. The extensions $j_\ast \mathcal M$ is given by
	$$  j_\ast \mathcal M = \mathcal D_{X} \otimes_{V_{0}^{D} \mathcal D_{X}} \mathcal M_{1} \simeq \mathcal D_{X}\cdot \mathcal M_{1}. $$
 The filtrations on $j_\ast \mathcal M$ are defined by the filtrations on $\mathcal D_{X}$ and $\mathcal M_{\alpha}$. To be precise, we have
	\begin{align*}
		& F_{p} j_\ast \mathcal M = \sum_{i} F_{p-i} \mathcal D_{X} \otimes F_{i} \mathcal M_{1}.
	\end{align*}
	Also, we can define filtered $D$-modules $\mathcal D_{X}\otimes_{V_{0}^{D} \mathcal D_{X}} \mathcal M_{\alpha}$ for each $\alpha \in \mathbb R^{r}$. The obvious map
	$$ \mathcal M_{\alpha} \to \mathcal D_{X} \otimes_{V_{0}^{D}\mathcal D_{X}} \mathcal M_{\alpha+ 1} $$
	obtained by the inclusion $\mathcal M_{\alpha} \hookrightarrow \mathcal M_{\alpha + 1}$ composed with $m \mapsto 1 \otimes m$ gives a morphism at the level of de Rham complexes.
	$$ \DR(\mathcal M_{\alpha}, \log D) \to \DR_{X} \left(\mathcal D_{X} \otimes_{V_{0}^{D}\mathcal D_{X}} \mathcal M_{\alpha+ 1}\right).$$
	The key point is that we can compute the de Rham complex on the right-hand side using the logarithmic de Rham complex.
	
	\begin{theorem} \cite{Sai90}
		For each $\alpha \in \mathbb R^{l}$, the morphism
		$$ \rm DR(\mathcal M_{\alpha}, \log D) \to \rm DR_{X}\left(\mathcal D_{X} \otimes_{V_{0}^{D}\mathcal D_{X}} \mathcal M_{\alpha + 1}  \right)$$
		is a filtered quasi-isomorphism.
	\end{theorem}

 If we apply this theorem for $\alpha = 0$ and take the graded pieces, we get the following isomorphisms:
	
	\begin{proposition} \label{prop-log-comparison}
		With the above notation, we get an isomorphism between the graded pieces of the de Rham complexes:
		\begin{align*}
			& \rm gr_{p}\rm DR_X(\mathcal M_{0}, \log D) \simeq_{qis} \rm gr_{p} \rm DR_{X} (j_\ast \mathcal M).
		\end{align*}
	\end{proposition}
	
This allows us to compute the de Rham complexes of $j_\ast \mathcal M$ in terms of the logarithmic de Rham complexes, which are easier to control. Now we prove the main Theorem \ref{main2}.

%
%

\begin{theorem} \label{prop-vanishing-singulardivisor}
		Let $X$ be a projective variety and $D$ an effective Cartier divisor on $X$. Suppose that $U = X \setminus D$ is smooth and consider the open inclusion $j \colon U \to X$. Let $\mathcal M$ be a variation of Hodge structures on $U$ that we identify with the corresponding pure Hodge module. For $j_{\ast} \mathcal M$, we have the following vanishing hold:
	\begin{align*}
		&\mathbb H^{p+q}(X, \rm gr_{-p}\rm DR_{X}(j_\ast \mathcal M)) = 0
		\end{align*}
		for all $p+q > k$.
	\end{theorem}
	
\begin{proof}

Consider a log resolution $\mu \colon Y \to X$ of the pair $(X, D)$, an isomorphism over $U$. Let $F = (\mu^{\ast} D)_{\mathrm{red}}$, which is an SNC divisor on $Y$. Write $F = \sum_{i=1}^{s} F_{i}$, we can not deduce that $F$ is the support of some $k$-ample divisor, but according to Remark \ref{keyrem}, we know Theorem \ref{snc-main} is still valid for $F$. Since $\mu$ is proper, we can compute the graded de Rham complex of $j_\ast \mathcal M$ as follows:
	\begin{align*}
		& \gr_{p} \DR_{X}(j_\ast \mathcal M) \simeq \mathbf R \mu_{\ast} \left( \gr_{p} \DR_{Y}(i_\ast \mathcal M) \right) \simeq \mathbf R \mu_{\ast} \left( \gr_{p} \DR_Y(\mathcal M_{0}, \log F) \right).
	\end{align*}
	The isomorphism on the right-hand side follows by a logarithmic comparison (Proposition \ref{prop-log-comparison}). Due to the projection formula, it is enough to show the vanishing of the following hyper cohomology:
	\begin{align*}
		& \mathbb H^{l}(Y, \gr_{p} \DR_Y(\mathcal M_{0}, \log F)).
	\end{align*}
Then we can use Theorem \ref{snc-main}  to get
	$$ \mathbb H^{l}(Y, \gr_{p} \DR_Y(\mathcal M_{0}, \log F)) = 0 \qquad \text{for all } l >k. $$
Noting that we shift $n$ in the de Rham complex.
\end{proof}

At last, we introduce the birational characterization of the Hodge ideal in Musta\c{t}\u{a}--Popa \cite{MP19}. Let $X$ be a smooth complex variety and $D$ be a reduced effective divisor on $X$. Given a log resolution $f: Y  \rightarrow X$ of the pair $(X, D)$ which is an isomorphism over $X\setminus D$, we 
define $E :=(f^{\ast}D)_{\rm red}$.
 We will see in \cite{MP19} that there is a filtered complex of right $f^{-1} \mathcal{D}_X$-modules 
$$0 \longrightarrow f^*\mathcal{D}_X \longrightarrow \Omega_Y^1(\log E) \otimes_{\shO_Y} f^* \mathcal{D}_X \longrightarrow \cdots$$
$$\cdots \longrightarrow \Omega_Y^{n-1}(\log E) \otimes_{\shO_Y} f^* \mathcal{D}_X \longrightarrow K_Y(E) \otimes_{\shO_Y} 
f^*\mathcal{D}_X \longrightarrow 0$$
which is exact except at the rightmost term, where the cohomology is $\omega_Y(*E) \otimes_{\mathcal{D}_Y}  \mathcal{D}_{Y\to X}$; 
here $ \mathcal{D}_{Y\to X}=f^*\mathcal{D}_X$ is the transfer module of $f$.
Denoting the above complex by $A^\bullet$, its filtration is provided by the subcomplexes $F_{k-n}A^\bullet$, for every $k\geq 0$, given by 
$$0 \longrightarrow f^* F_{k-n} \mathcal{D}_X \longrightarrow \Omega_Y^1(\log E) \otimes_{\shO_Y} f^* F_{k-n+1} \mathcal{D}_X \longrightarrow \cdots $$
$$\cdots \longrightarrow \Omega_Y^{n-1} (\log E) \otimes{\shO_Y} f^* F_{k-1} \mathcal{D}_X \longrightarrow  K_Y(E) \otimes_{\shO_Y} f^* F_k \mathcal{D}_X\longrightarrow 0.$$

Musta\c{t}\u{a}--Popa character the \emph{$k$-th Hodge ideal} $I_k(D)$ associated to $D$ by the formula
\begin{equation} \label{defnH}
K_X \big( (k +1)D \big) \otimes I_k(D) = {\rm Im} \left[R^0 f_* F_{k-n}A^\bullet \longrightarrow R^0 f_* A^\bullet \right].
\end{equation}
Not only the image is contained in $\omega_X \big( (k +1)D \big)$, but also this definition is independent of the choice of log resolution, and it indeed coincides with the ideals defined by Saito's Hodge filtration.
The $0$-th Hodge ideal belongs to a class of ideal sheaves that is quite well understood, and has proved to be extremely useful; 
it is not hard to show that
$$\mathcal{I}_0 (D) = \mathcal{I} \big(X, (1- \epsilon)D\big),$$
the multiplier ideal associated to the $\mathbb Q$-divisor $(1-\epsilon)D$ with $0 < \epsilon \ll 1$. 

\begin{lemma} \cite[Proposition $10.1$]{MP19}\label{ideal-I0}
We have 
$$ \mathcal 
I_0 (D) = \mathcal J \big(X, (1-\epsilon)D \big),$$
the multiplier ideal associated to the $\mathbb Q$-divisor $(1-\epsilon)D$ on $X$, for any $0 < \epsilon \ll 1$.
\end{lemma}
\begin{proof}
Recall that $F_{-n} A^{\bullet} = K_Y (E)$, hence by formula \eqref{defnH}
$$F_0 K_X(*D)={\rm Im}\big(f_*K_Y(E)\to K_X(D)\big)=f_*\shO_Y(K_{Y/X} + E - f^*D)\otimes K_X(D).$$
Therefore the statement to be proved is that 
$$f_* \shO_Y (K_{Y/X} + E - f^*D)  = \mathcal I \big(X, (1-\epsilon)D \big).$$
On the other hand, the right-hand side is by definition 
$$f_* \shO_Y \big(K_{Y/X} - \lfloor (1-\epsilon)f^*D\rfloor \big)=f_*\shO_Y\big(K_{Y/X}+(f^*D)_{\rm red}-f^*D\big),$$
which implies the desired equality.
\end{proof}

In particular, $\mathcal{I}_0 (D) = \shO_X$ if and only if the pair $(X, D)$ is log-canonical. Hodge ideals can be computed concretely when $D$ is a simple normal crossing divisor, see \cite[Proposition $8.2$]{MP19}. In particular, if $D$ is smooth, then $I_k (D) = \mathcal O_X$ for all $k \ge 0$, which corresponds to equality between the Hodge filtration and the pole order filtration.

\end{document}